\documentclass[11pt,leqno]{amsart}

\usepackage[a4paper,left=1.5cm,top=2cm,right=1.5cm,bottom=2cm]{geometry}
\usepackage[T1]{fontenc}
\usepackage{amsmath,amssymb,mathtools}
\usepackage[hidelinks]{hyperref}

\newtheorem{proposition}{Proposition}
\newtheorem{corollary}[proposition]{Corollary}
\newtheorem{lemma}[proposition]{Lemma}
\theoremstyle{remark}
\newtheorem{remark}[proposition]{Remark}

\DeclareMathOperator{\conv}{conv}

\title[Erratum]{Erratum and a counterexample to
\textquotedblleft Stability of diametral diameter two properties\textquotedblright}

\author{Johann Langemets}
\address{Institute of Mathematics and Statistics, University of Tartu,
Narva mnt 18, 51009 Tartu, Estonia}
\email{johann.langemets@ut.ee}
\thanks{This work was supported by the Estonian Research Council grant
(PRG2545) and partially supported by the grant PID2025-167660NB-I00
funded by MICIU/AEI/10.13039/501100011033 and ERDF/EU}
\urladdr{\url{https://www.johannlangemets.com/}}

\date{}

\begin{document}

\begin{abstract}
In Proposition~3.5 of our paper
[J.~Langemets and K.~Pirk,
\emph{Stability of diametral diameter two properties},
RACSAM (2021)], we claimed that
the convex DLD2P lifts from an $M$-ideal to its superspace.
The proof contains a quantifier error. We show that its valid part
yields the following replacement: for every $\varepsilon>0$, every
slice of the unit ball of the superspace contains a
$(2-\varepsilon)$ $\Delta$-point of norm arbitrarily close to one.
In particular, the unit ball is the closed convex hull of these points
at each fixed scale, and the superspace has the local diameter two
property. Finally, we prove that the original claim was false, we construct Banach spaces $Y$ and $X$ such that $Y$ has the convex DLD2P, $Y$ is an $M$-ideal in $X$, yet $X$ fails the convex DLD2P. 
\end{abstract}

\maketitle

\section{Introduction and the error}

Throughout, all Banach spaces are real and nonzero.
We denote the closed unit ball and unit sphere of a Banach space $X$
by $B_X$ and $S_X$, respectively, and its dual space by $X^*$.
For $x^*\in S_{X^*}$ and $\alpha>0$, a slice of $B_X$ is denoted by
\[
S(B_X,x^*,\alpha):=\{x\in B_X:x^*(x)>1-\alpha\}.
\]
For $x\in B_X$ and $\varepsilon>0$, set
\[
\Delta^X_\varepsilon(x)
:=\{y\in B_X:\|x-y\|\geq 2-\varepsilon\}.
\]
Here, unlike in \cite{LangemetsPirk2021}, we allow the centre $x$
to be any element of $B_X$; this agrees with the terminology for
almost $\Delta$-points introduced in
\cite[Section~6]{AbrahamsenLimaMartinyPerreau2022}.

An element $x\in S_X$ is a $\Delta$-point if
\[
x\in\overline{\conv}\,\Delta^X_\varepsilon(x)
\qquad\text{for every }\varepsilon>0.
\]
We denote the set of all $\Delta$-points of $X$ by $\Delta_X$.
Recall from \cite{AbrahamsenHallerLimaPirk2020} that $X$ has the
\emph{convex diametral local diameter two property} (convex DLD2P) if
\[
B_X=\overline{\conv}\,\Delta_X.
\]
Equivalently, every slice of $B_X$ intersects $\Delta_X$.
The space $X$ has the \emph{local diameter two property} (LD2P) if every
slice of $B_X$ has diameter two.

Proposition~3.5 of \cite{LangemetsPirk2021} asserted that if $Y$
is an $M$-ideal in $X$ and $Y$ has the convex DLD2P, then $X$
has the convex DLD2P. For each choice of a tolerance, the proof
constructs a point in a prescribed slice which is close to the
convex hull of points at distance almost two from it. The
constructed point, however, depends on the tolerance. Thus the
proof provides a point separately for each tolerance, whereas
establishing a $\Delta$-point requires a single point that satisfies
the defining condition for every tolerance. The scaling argument
in the last paragraph of that proof does not resolve this
interchange of quantifiers.

Proposition~\ref{prop:replacement} below gives a valid replacement,
and Proposition~\ref{prop:counterexample} shows that the original
assertion is false. Proposition~3.5 of \cite{LangemetsPirk2021}
and the accompanying assertions in its introduction and Section~3
that the convex DLD2P lifts from an $M$-ideal to its superspace
must therefore be withdrawn. No other result in \cite{LangemetsPirk2021} is affected by this error.

\section{A valid replacement for Proposition 3.5}

Following \cite[Section~6]{AbrahamsenLimaMartinyPerreau2022},
an element $x\in B_X$ is called a $(2-\varepsilon)$ $\Delta$-point
if $x\in\overline{\conv}\,\Delta^X_\varepsilon(x)$.
For $\varepsilon>0$, put
\[
\Delta_X^{(\varepsilon)}
:=\{x\in B_X:x\in\overline{\conv}\,\Delta^X_\varepsilon(x)\}.
\]
A Banach space \emph{admits almost $\Delta$-points} if
$\Delta_X^{(\varepsilon)}\neq\varnothing$ for every $\varepsilon>0$.

The original argument yields a conclusion at each fixed positive scale. Although the center may depend on the scale, it can be chosen to be exactly a finite convex combination of points almost diametral to it.

\begin{proposition}\label{prop:replacement}
Let $X$ be a Banach space and let $Y$ be an $M$-ideal in $X$.
If $Y$ has the convex DLD2P, then, for every slice $S$ of $B_X$
and every $\varepsilon,\delta>0$, there exist
$v,u_1,\ldots,u_n\in B_X$ and
$\lambda_1,\ldots,\lambda_n\geq0$ with
$\sum_{i=1}^n\lambda_i=1$ such that
\[
v\in S,\qquad \|v\|>1-\delta,\qquad
v=\sum_{i=1}^n\lambda_i u_i,\qquad
\|v-u_i\|>2-\varepsilon\quad(i=1,\ldots,n).
\]
In particular, $v\in\conv\Delta^X_\varepsilon(v)$.
\end{proposition}

\begin{proof}
Let $S=S(B_X,x^*,\alpha)$, where $x^*\in S_{X^*}$ and $\alpha>0$.
Let $P\colon X^*\to X^*$ be the $L$-projection with
$\ker P=Y^\perp$. Restriction to $Y$ identifies $\operatorname{ran}P$
isometrically with $Y^*$. The $L$-decomposition gives
\[
\|Px^*\|+\|x^*-Px^*\|=1.
\]

Set $\widetilde\varepsilon:=\min\{\varepsilon,\delta\}$ and choose
$\gamma,\theta,\rho>0$ so small that
\begin{equation}\label{eq:parameters}
3\gamma+\frac{\rho}{1+\gamma}<\alpha
\qquad\text{and}\qquad
\frac{2-\theta-\rho}{1+\gamma}>2-\widetilde\varepsilon.
\end{equation}

Suppose first that $Px^*\neq0$, and define
\[
y^*:=\frac{(Px^*)|_Y}{\|Px^*\|}\in S_{Y^*},
\qquad
\beta:=\frac{\gamma(1-\|Px^*\|)+\gamma^2}{\|Px^*\|}>0.
\]
Since $Y$ has the convex DLD2P, we may choose
\[
y\in S(B_Y,y^*,\beta)\cap\Delta_Y.
\]
The choice of $\beta$ gives
\begin{equation}\label{eq:Py}
Px^*(y)>(\|Px^*\|-\gamma)(1+\gamma).
\end{equation}
If $Px^*=0$, choose any $y\in\Delta_Y$; then
\eqref{eq:Py} holds automatically.

Since $y$ is a $\Delta$-point, there exist
$y_1,\ldots,y_n\in B_Y$ and convex coefficients
$\lambda_1,\ldots,\lambda_n$ such that, with
$\bar y:=\sum_{i=1}^n\lambda_i y_i$,
\begin{equation}\label{eq:yapprox}
\|y-\bar y\|<\rho
\qquad\text{and}\qquad
\|y-y_i\|\geq2-\theta\quad(i=1,\ldots,n).
\end{equation}
Choose $x\in B_X$ such that
\begin{equation}\label{eq:x0}
(x^*-Px^*)(x)>
(\|x^*-Px^*\|-\gamma)(1+\gamma).
\end{equation}
This is possible because $(b-\gamma)(1+\gamma)<b$
for every $b\in[0,1]$.

By the basic inequality for $M$-ideals
\cite[Proposition~2.3]{Werner1994}, there exists $z\in Y$ such that
\begin{equation}\label{eq:basic}
\begin{aligned}
\|y+x-z\|&<1+\gamma,\\
\|y_i+x-z\|&<1+\gamma\quad(i=1,\ldots,n),\\
|Px^*(x-z)|&<\gamma.
\end{aligned}
\end{equation}
Define
\[
u:=\frac{y+x-z}{1+\gamma},\qquad
u_i:=\frac{y_i+x-z}{1+\gamma}\quad(i=1,\ldots,n),
\qquad v:=\sum_{i=1}^n\lambda_i u_i.
\]
Then $u,u_i,v\in B_X$. From \eqref{eq:Py}, \eqref{eq:x0},
and \eqref{eq:basic}, we obtain
\begin{align*}
x^*(u)
&=\frac{Px^*(y)+(x^*-Px^*)(x)+Px^*(x-z)}{1+\gamma}\\
&>\frac{(\|Px^*\|-\gamma)(1+\gamma)
       +(\|x^*-Px^*\|-\gamma)(1+\gamma)-\gamma}{1+\gamma}\\
&=1-2\gamma-\frac{\gamma}{1+\gamma}>1-3\gamma.
\end{align*}
Moreover,
\[
\|v-u\|=\frac{\|\bar y-y\|}{1+\gamma}
<\frac{\rho}{1+\gamma}.
\]
Consequently, by \eqref{eq:parameters},
\[
x^*(v)>1-3\gamma-\frac{\rho}{1+\gamma}>1-\alpha,
\]
and hence $v\in S$. Finally, for every $i$,
\begin{align*}
\|v-u_i\|
&=\frac{\|\bar y-y_i\|}{1+\gamma}\\
&\geq\frac{\|y-y_i\|-\|y-\bar y\|}{1+\gamma}\\
&>\frac{2-\theta-\rho}{1+\gamma}
>2-\widetilde\varepsilon\geq2-\varepsilon.
\end{align*}
Thus $v=\sum_i\lambda_i u_i\in\conv\Delta^X_\varepsilon(v)$.
Also,
\[
\|v\|\geq\|v-u_i\|-\|u_i\|
>1-\widetilde\varepsilon\geq1-\delta.
\qedhere
\]
\end{proof}

\begin{corollary}\label{cor:almost}
Let $X$ be a Banach space and let $Y$ be an $M$-ideal in $X$.
If $Y$ has the convex DLD2P, then
\[
B_X=\overline{\conv}\,\Delta_X^{(\varepsilon)}
\qquad(\varepsilon>0).
\]
In particular, $X$ admits almost $\Delta$-points and has the LD2P.
\end{corollary}

\begin{proof}
For each $\varepsilon>0$, Proposition~\ref{prop:replacement}
shows that every slice of $B_X$ meets $\Delta_X^{(\varepsilon)}$.
The displayed equality follows by Hahn--Banach separation,
and the existence of almost $\Delta$-points follows immediately.

The space $X$ has the local diameter two property, because the convex DLD2P of $Y$ implies the LD2P of $Y$ \cite[Proposition~5.2]{AbrahamsenHallerLimaPirk2020} and the LD2P lifts from $Y$ to $X$ \cite[Proposition~4]{HallerLangemets2014}.
\end{proof}

\section{A counterexample to Proposition~3.5}

We first show that a single $\Delta$-point of a Banach space
suffices for its $c_0$-sum to have the convex DLD2P.

\begin{lemma}\label{lem:c0-four-delta}
Let $X$ be a Banach space containing a $\Delta$-point.
Then every element of $B_{c_0(X)}$ is the average of four
$\Delta$-points of $c_0(X)$. In particular,
\[
B_{c_0(X)}=\conv(\Delta_{c_0(X)}).
\]
\end{lemma}

\begin{proof}
We first note that if $z=(x_k)\in S_{c_0(X)}$ and
$x_m\in\Delta_X$ for some $m$, then $z\in\Delta_{c_0(X)}$.
Indeed, fix $\varepsilon,\eta>0$ and choose a convex combination
$\sum_i\lambda_i a_i$ with
\[
a_i\in\Delta_\varepsilon^X(x_m),
\qquad
\left\|x_m-\sum_i\lambda_i a_i\right\|<\eta.
\]
Let $z_i$ agree with $z$ except that its $m$th coordinate is $a_i$.
Then $z_i\in B_{c_0(X)}$,
$\|z-z_i\|\geq2-\varepsilon$, and
$\|z-\sum_i\lambda_i z_i\|<\eta$, proving the observation.
This is also a special case of
\cite[Theorem~4.1(2)(d)]{LeeRoldanTag}.

Fix $\widetilde x\in\Delta_X$ and let $z=(x_k)\in B_{c_0(X)}$.
Choose distinct indices $m,n$ such that
$\|x_m\|,\|x_n\|\leq1/2$.
Define $z^{(1)},\ldots,z^{(4)}$ to agree with $z$ outside
$\{m,n\}$ and prescribe their remaining coordinates by
\[
\begin{array}{c|cc}
 &m&n\\ \hline
z^{(1)}&\widetilde x&2x_n\\
z^{(2)}&-\widetilde x&2x_n\\
z^{(3)}&2x_m&\widetilde x\\
z^{(4)}&2x_m&-\widetilde x
\end{array}
\]
All four sequences belong to $S_{c_0(X)}$.
Since $\widetilde x$ and $-\widetilde x$ are $\Delta$-points,
the preceding observation gives $z^{(i)}\in\Delta_{c_0(X)}$
for every $i$. Finally,
\[
z=\frac14\sum_{i=1}^4z^{(i)}.
\qedhere
\]
\end{proof}

The next observation will be used to exclude $\Delta$-points from a
slice.

\begin{lemma}\label{lem:isolated-norming}
Let \(X\) be a Banach space and let
\(\mathcal S\subset B_{X^*}\) be a symmetric norming set, that is,
\[
    \|z\|=\sup_{f\in\mathcal S}f(z)
    \qquad (z\in X).
\]
Suppose that \(x\in S_X\), that
\[
    A(x):=\{f\in\mathcal S:f(x)=1\}
\]
is finite and nonempty, and that
\[
    \sup\{f(x):f\in\mathcal S\setminus A(x)\}<1.
\]
Then \(x\) is not a \(\Delta\)-point.
\end{lemma}

\begin{proof}
Write $A(x)=\{f_1,\ldots,f_n\}$
and set
\[
    g:=\frac1n\sum_{i=1}^n f_i.
\]
Observe that \(g\in S_{X^*}\) because $g(x)=1$. Put
\[
    \gamma
    :=\sup\{f(x):f\in\mathcal S\setminus A(x)\}<1
\]
and consider the slice $S(B_X,g,\frac{1}{2n})$. Clearly \(x\in S(B_X,g,\frac{1}{2n})\). 

Let \(y\in S(B_X,g,\frac{1}{2n})\). For every \(i\in\{1,\ldots,n\}\),
using \(f_j(y)\leq 1\) for \(j\neq i\), we obtain
\[
\begin{aligned}
    f_i(y)
      &=ng(y)-\sum_{j\neq i}f_j(y)\\
      &>n\left(1-\frac{1}{2n}\right)-(n-1)
       =\frac12.
\end{aligned}
\]
Consequently, if \(f\in A(x)\), then
\[
    f(x-y)=1-f(y)<\frac12.
\]
If \(f\in\mathcal S\setminus A(x)\), then \(f(x)\leq\gamma\), and hence
\[
    f(x-y)=f(x)-f(y)
       \leq \gamma+|f(y)|
       \leq \gamma+1.
\]
Because \(\mathcal S\) is norming, it follows that
\[
\begin{aligned}
    \|x-y\|
       &=\sup_{f\in\mathcal S}f(x-y)\\
       &\leq \max\left\{\frac12,1+\gamma\right\}
       <2.
\end{aligned}
\]
Thus
\[
    \sup_{y\in S(B_X,g,\frac{1}{2n})}\|x-y\|
       \leq \max\left\{\frac12,1+\gamma\right\}<2.
\]
We have therefore found a slice \(S(B_X,g,\frac{1}{2n})\) of \(B_X\) containing \(x\) on
which the distance from \(x\) is uniformly smaller than \(2\).
Hence \(x\) is not a \(\Delta\)-point.
\end{proof}

We are now ready to construct the promised counterexample.

\begin{proposition}\label{prop:counterexample}
There exist a Banach space $X$ and a codimension-one $M$-ideal
$Y$ in $X$ such that
\[
\Delta_X=\Delta_Y,\qquad
B_Y=\conv(\Delta_Y)=\conv(\Delta_X)\subsetneq B_X,
\]
where $Y$ is identified with its isometric copy in $X$.
In particular, $Y$ has the convex DLD2P, whereas $X$ does not.
\end{proposition}

\begin{proof}
We will construct these spaces $Y$ and $X$ and verify their properties in several steps. 

\textbf{Step 1. Constructing the spaces $Y$ and $X$}.
Let $c$ and $c_0$ denote the spaces of convergent and null
real sequences. On $E=c\times c_0$ define
\[
\|(w,z)\|_E
=\max\left\{
\|w\|_\infty,\ \|z\|_\infty,\
\left|\lim_n w_n\right|+\sup_n\frac{|z_n|}{n}
\right\}.
\]
This is a norm equivalent to the usual product norm.
Put $Y=c_0(E)$ and write
\[
y=((w^j,z^j))_{j=1}^\infty,\qquad
\ell_j(y)=\lim_n w_n^j\quad(y\in Y).
\]
On the vector space $X=Y\oplus\mathbb R$ define
\[
\begin{split}
\|(y,t)\|_X=\max\Biggl\{
&|t|,\ \sup_j\|w^j\|_\infty,\
\sup_{j,n}\left|z_n^j+\frac t2\right|,\\
&\sup_{j,n}\left(
|\ell_j(y)|+\frac1n\left|z_n^j+\frac t2\right|
\right)\Biggr\}.
\end{split}
\]
This expression is a norm and restricts to the given norm on $Y$.
If $p(y,t)=\max\{\|y\|_Y,|t|\}$, then
\[
\frac23p(y,t)\leq\|(y,t)\|_X\leq\frac32p(y,t).
\]
Indeed, the shift by $t/2$ changes each relevant coordinate bound
by at most $|t|/2$, giving both
$\|(y,t)\|_X\leq p(y,t)+|t|/2$ and
$\|y\|_Y\leq\|(y,t)\|_X+|t|/2$.
Thus $X$ is a Banach space and $Y=Y\oplus\{0\}$ is an isometric
codimension-one subspace.

\medskip
\textbf{Step 2. The equality $B_Y=\conv(\Delta_Y)$}.
Let $e=(\mathbf1,0)\in S_E$, where $\mathbf1=(1,1,\ldots)$,
and denote the standard unit sequences by $p_k$.
The points
\[
u_k=(\mathbf1-2p_k,0)
\]
belong to $S_E$ and satisfy
\[
\|e-u_k\|_E=2,\qquad
\left\|e-\frac1N\sum_{k=1}^Nu_k\right\|_E=\frac2N.
\]
Hence $e\in\Delta_E$, and, thus Lemma~\ref{lem:c0-four-delta} gives
\[
B_Y=\conv(\Delta_Y).
\]

\medskip
\textbf{Step 3. $Y$ is an $M$-ideal in $X$}.
We verify the restricted three-ball property
\cite[Theorem~I.2.2(iv)]{HarmandWernerWerner1993}: for all
$y_1,y_2,y_3\in B_Y$, $x\in B_X$, and $\varepsilon>0$,
there exists $\widetilde y\in Y$ such that
\[
\|x+y_i-\widetilde y\|_X\leq1+\varepsilon
\qquad(i=1,2,3).
\]

Let
\[
x=(y_0,t)\in B_X,\qquad
y_i=((w_i^j,z_i^j))_{j=1}^\infty\in B_Y
\quad(i=1,2,3),
\]
and let $\varepsilon>0$. Write
\[
y_0=((w_0^j,z_0^j))_{j=1}^\infty,\qquad
\ell_{i,j}=\lim_n w_{i,n}^j.
\]
By the definitions of the norms, we have, for
$i=1,2,3$ and $j,n\in\mathbb N$,
\begin{equation}\label{eq:three-ball-basic}
\begin{gathered}
|t|\leq1,\qquad
\sup_j\|w_i^j\|_\infty\leq1,\qquad
|z_{i,n}^j|\leq1,\\
|\ell_{i,j}|+\frac{|z_{i,n}^j|}{n}
\leq\|(w_i^j,z_i^j)\|_E\leq1.
\end{gathered}
\end{equation}

Since $y_i\in c_0(E)$ and each $z_i^j\in c_0$, we may choose
$m\in\mathbb N$ sufficiently large that, for $i=1,2,3$,
\begin{align}
\sup_{j>m}\|(w_i^j,z_i^j)\|_E
&\leq\varepsilon,
\label{eq:three-ball-row-tail}\\
\sup_{\max\{j,n\}>m}|z_{i,n}^j|
&\leq\varepsilon,
\label{eq:three-ball-square-tail}\\
\frac1{2(m+1)}
&\leq\varepsilon.
\label{eq:three-ball-weight-tail}
\end{align}
Indeed, the arrays $(z_{i,n}^j)_{j,n}$ are uniformly small
outside sufficiently large finite squares, because
$z_i^j\in c_0$ and $\|z_i^j\|_\infty\to0$ as $j\to\infty$.

Define $\widetilde y=((\widetilde w^j,\widetilde z^j))_j$ by
\begin{equation}\label{eq:three-ball-correction}
\widetilde w^j=w_0^j,\qquad
\widetilde z_n^j
=z_{0,n}^j+\frac t2\mathbf1_{\{j\leq m,\ n\leq m\}}.
\end{equation}
This is a finite-coordinate perturbation of $y_0$, so
$\widetilde y\in Y$.

Fix $i\in\{1,2,3\}$. By \eqref{eq:three-ball-correction},
the $w$-rows of $x+y_i-\widetilde y$ are $w_i^j$.
Its scalar coordinate remains $t$, so its shifted
$z$-coordinates in the norm are
\begin{equation}\label{eq:three-ball-shifted}
\begin{aligned}
q_{i,j,n}
&:=z_{0,n}^j+z_{i,n}^j-\widetilde z_n^j+\frac t2\\
&=
\begin{cases}
z_{i,n}^j,&j,n\leq m,\\
z_{i,n}^j+t/2,&\max\{j,n\}>m.
\end{cases}
\end{aligned}
\end{equation}
Consequently,
\begin{equation}\label{eq:three-ball-residual-norm}
\begin{split}
\|x+y_i-\widetilde y\|_X
=\max\Biggl\{
&|t|,\ \sup_j\|w_i^j\|_\infty,\
\sup_{j,n}|q_{i,j,n}|,\\
&\sup_{j,n}\left(
|\ell_{i,j}|+\frac{|q_{i,j,n}|}{n}
\right)\Biggr\}.
\end{split}
\end{equation}

The first two terms in \eqref{eq:three-ball-residual-norm}
are at most $1$ by \eqref{eq:three-ball-basic}.
If $j,n\leq m$, then \eqref{eq:three-ball-shifted} and
\eqref{eq:three-ball-basic} give
\[
|q_{i,j,n}|\leq1,\qquad
|\ell_{i,j}|+\frac{|q_{i,j,n}|}{n}\leq1.
\]
Outside this square, \eqref{eq:three-ball-shifted},
\eqref{eq:three-ball-square-tail}, and the bound $|t|\leq1$
from \eqref{eq:three-ball-basic} yield
\[
|q_{i,j,n}|
\leq|z_{i,n}^j|+\frac{|t|}{2}
\leq\varepsilon+\frac12
\leq1+\varepsilon.
\]

For the mixed term, suppose first that $j>m$.
Using \eqref{eq:three-ball-shifted},
\eqref{eq:three-ball-basic}, and
\eqref{eq:three-ball-row-tail}, we obtain
\[
\begin{aligned}
|\ell_{i,j}|+\frac{|q_{i,j,n}|}{n}
&\leq|\ell_{i,j}|+\frac{|z_{i,n}^j|}{n}
+\frac{|t|}{2n}\\
&\leq\|(w_i^j,z_i^j)\|_E+\frac12\\
&\leq\varepsilon+\frac12
\leq1+\varepsilon.
\end{aligned}
\]
If $j\leq m<n$, then $n\geq m+1$. Hence
\eqref{eq:three-ball-shifted},
\eqref{eq:three-ball-basic}, and
\eqref{eq:three-ball-weight-tail} give
\[
\begin{aligned}
|\ell_{i,j}|+\frac{|q_{i,j,n}|}{n}
&\leq|\ell_{i,j}|+\frac{|z_{i,n}^j|}{n}
+\frac{|t|}{2n}\\
&\leq1+\frac1{2(m+1)}
\leq1+\varepsilon.
\end{aligned}
\]

Combining these estimates with
\eqref{eq:three-ball-residual-norm}, we conclude that
\[
\|x+y_i-\widetilde y\|_X\leq1+\varepsilon
\qquad(i=1,2,3).
\]
This proves the restricted three-ball property.
Therefore $Y$ is an $M$-ideal in $X$.

\medskip
\textbf{Step 4. $X$ does not have the convex DLD2P}.
We prove that $\Delta_X\subseteq Y$ by applying
Lemma~\ref{lem:isolated-norming}.

Fix $x=(y,t)\in S_X$ with $t\neq0$, and write
\[
y=((w^j,z^j))_{j=1}^\infty,\qquad
\ell_j=\lim_n w_n^j,\qquad
L=\sup_j|\ell_j|.
\]
The definition of the norm gives
\begin{equation}\label{eq:step4-mixed-bound}
|\ell_j|+\frac{|z_n^j+t/2|}{n}\leq1
\qquad(j,n\in\mathbb N).
\end{equation}
If $|\ell_j|=1$ for some $j$, then
\eqref{eq:step4-mixed-bound} forces $z_n^j=-t/2$ for every
$n$, contradicting $z^j\in c_0$. Thus $|\ell_j|<1$ for every
$j$. Since $\ell_j\to0$, it follows that $L<1$.

The condition $y\in c_0(E)$ also implies
\[
\sup_j|w_n^j-\ell_j|\longrightarrow0,
\qquad
\sup_j|z_n^j|\longrightarrow0.
\]
Indeed, the outer tail is uniformly small, while the remaining
rows are finite in number. Consequently, there exist
$m\in\mathbb N$ and $0\leq\rho<1$ such that
\begin{equation}\label{eq:step4-tail-bound}
\max\left\{
|w_n^j|,\ |z_n^j+t/2|,\
|\ell_j|+\frac{|z_n^j+t/2|}{n}
\right\}\leq\rho
\qquad\text{if }\max\{j,n\}>m.
\end{equation}
For the mixed term, the outer tail is controlled by
\[
|\ell_j|+\frac{|z_n^j+t/2|}{n}
\leq\|(w^j,z^j)\|_E+\frac12,
\]
whereas the inner tail is controlled by
\[
|\ell_j|+\frac{|z_n^j+t/2|}{n}
\leq L+\frac1n.
\]
Here we used $|t|\leq1$ and $|z_n^j+t/2|\leq1$.

We now specify a symmetric norming set for $X$.
For $(v,s)\in X$, where $v=((a^j,b^j))_j$, define the
linear functionals
\[
\begin{aligned}
T(v,s)&=s,\qquad
W_{j,n}(v,s)=a_n^j,\qquad
Z_{j,n}(v,s)=b_n^j+\frac s2,\qquad
\Lambda_j(v,s)=\lim_n a_n^j.
\end{aligned}
\]
Let
\[
\begin{aligned}
\mathcal S
={}\{\pm T\}
&{}\cup
\{\pm W_{j,n},\ \pm Z_{j,n}:j,n\in\mathbb N\}
{}\cup
\left\{
\sigma\Lambda_j+\frac{\tau}{n}Z_{j,n}:
j,n\in\mathbb N,\ \sigma,\tau\in\{-1,1\}
\right\}.
\end{aligned}
\]
Each member of $\mathcal S$ belongs to $B_{X^*}$, and
$\mathcal S$ is symmetric. Moreover, the identity
\[
|a|+|b|
=\max{\{|a+b|, |a-b|\}}
\qquad(a,b\in\mathbb R)
\]
and the definition of the norm show that
\begin{equation}\label{eq:step4-norming-set}
\|u\|_X=\sup_{f\in\mathcal S}f(u)
\qquad(u\in X).
\end{equation}

Let $\mathcal S_m$ be the finite subset obtained by restricting
$j,n$ to $\{1,\ldots,m\}$ in the definition of $\mathcal S$,
while retaining $\pm T$. By \eqref{eq:step4-tail-bound},
\begin{equation}\label{eq:step4-functional-tail}
|f(x)|\leq\rho
\qquad(f\in\mathcal S\setminus\mathcal S_m).
\end{equation}
It follows that
\[
A(x):=\{f\in\mathcal S:f(x)=1\}
\subseteq\mathcal S_m,
\]
so $A(x)$ is finite. It is also nonempty: otherwise, since
$\mathcal S_m$ is finite and $\rho<1$,
\eqref{eq:step4-functional-tail} would imply
$\sup_{f\in\mathcal S}f(x)<1$, contrary to
\eqref{eq:step4-norming-set} and $\|x\|_X=1$.

Finally, put
\[
\gamma:=
\max\left(
\{\rho\}\cup
\{f(x):f\in\mathcal S_m\setminus A(x)\}
\right).
\]
Since $\mathcal S_m\setminus A(x)$ is finite and each of its
members takes a value strictly smaller than $1$ at $x$,
we have $\gamma<1$. Together with
\eqref{eq:step4-functional-tail}, this gives
\[
\sup\{f(x):f\in\mathcal S\setminus A(x)\}
\leq\gamma<1.
\]
Lemma~\ref{lem:isolated-norming} therefore yields
$x\notin\Delta_X$. Since $x=(y,t)\in S_X$ with $t\neq0$
was arbitrary, we conclude that $\Delta_X\subseteq Y$.

On the other hand, the proof of \cite[Theorem~3.2(b)]{LangemetsPirk2021} yields
$\Delta_X\cap Y\subseteq\Delta_Y$ whenever $Y$ is an $M$-ideal
in $X$. Together with $\Delta_X\subseteq Y$ and the isometric
inclusion $\Delta_Y\subseteq\Delta_X$, this gives
$\Delta_X=\Delta_Y$. Hence, Step 2 implies
\[
B_Y=\operatorname{conv}(\Delta_Y)
\subseteq\operatorname{conv}(\Delta_X)
\subseteq B_Y.
\]
Since $B_Y$ is closed in $X$ and $(0,1)\in S_X\setminus Y$,
we obtain
\[
\overline{\operatorname{conv}}(\Delta_X)
=\operatorname{conv}(\Delta_X)
=B_Y\subsetneq B_X.
\]
Thus $Y$ has the convex DLD2P, whereas $X$ does not.
\end{proof}

Taken together, Corollary~\ref{cor:almost} and
Proposition~\ref{prop:counterexample} show that approximation
at every positive scale does not guarantee the convex DLD2P.
The exact property can fail to pass from an $M$-ideal to its
ambient space even in codimension one. The following remark
makes this distinction precise: taking closed convex hulls
need not commute with intersecting the sets of approximate
$\Delta$-points over all positive scales.

\begin{remark}\label{rem:sharpness}
For every Banach space $X$,
\[
\bigcap_{\varepsilon>0}\Delta_X^{(\varepsilon)}=\Delta_X.
\]
Indeed, membership in $\Delta_X^{(\varepsilon)}$ implies
$\|x\|\geq1-\varepsilon$, so membership for every
$\varepsilon>0$ forces $\|x\|=1$.

For the space in Proposition~\ref{prop:counterexample},
Corollary~\ref{cor:almost} gives
\[
B_X=\overline{\conv}\,\Delta_X^{(\varepsilon)}
\qquad(\varepsilon>0),
\]
whereas
\[
\overline{\conv}\left(
\bigcap_{\varepsilon>0}\Delta_X^{(\varepsilon)}
\right)
=\overline{\conv}(\Delta_X)=B_Y\subsetneq B_X.
\]
Thus the conclusion at each fixed scale cannot be strengthened
to the convex DLD2P.
\end{remark}

\section*{AI disclosure statement}
The author used OpenAI ChatGPT through a ChatGPT for Academic Researchers workspace. The model ChatGPT 6.0 Astra was used to develop the initial idea for constructing a counterexample to Proposition~3.5 of \cite{LangemetsPirk2021}. In Proposition~\ref{prop:counterexample}, we present a simpler counterexample than the initial counterexample constructed by AI. All outputs were critically reviewed and independently verified by the author, who take full responsibility for the final content.

\section*{Acknowledgements}
The author is grateful to Juan Guerrero-Viu for spotting the error in the proof of Proposition~3.5 of \cite{LangemetsPirk2021}
and informing us about it.

\end{document}